\pdfoutput=1
\documentclass[leqno,a4paper,11pt]{article}
\usepackage{amsmath,amsthm,amssymb}
\usepackage[utf8]{inputenc}
\usepackage[T1]{fontenc}

\usepackage[a4paper]{geometry}

\providecommand{\noopsort}[1]{} 

\usepackage{enumerate}
\usepackage{bbm}
\usepackage{todonotes}
\usepackage{mathrsfs}
\usepackage{mathabx}
\usepackage{hyperref}

\usepackage{mdframed}

\usepackage[sort&compress,numbers]{natbib}

\usepackage{url}
\makeatletter
\g@addto@macro{\UrlBreaks}{\UrlOrds}
\makeatother

\usepackage{nicefrac}
\usepackage{tikz}
\usetikzlibrary{matrix}

\usepackage{todonotes}

\usepackage{cleveref}
\crefname{Th}{Theorem}{Theorems}
\crefname{Prop}{Proposition}{Propositions}
\crefname{Lemma}{Lemma}{Lemmas}
\crefname{Cor}{Corollary}{Corollaries}
\crefname{Remark}{Remark}{Remarks}
\crefname{Def}{Definition}{Definitions}
\crefname{Example}{Example}{Examples}
\crefname{section}{Section}{Sections}

\newtheorem{Th}{Theorem}[section]
\newtheorem{Prop}[Th]{Proposition}
\newtheorem{Lemma}[Th]{Lemma}
\newtheorem{Cor}[Th]{Corollary}
\theoremstyle{definition}
\newtheorem{Remark}[Th]{Remark}
\newtheorem{Def}{Definition}[section]
\newtheorem{Example}{Example}[section]

\newcommand{\beq}{\begin{equation}}
\newcommand{\eeq}{\end{equation}}

\def\scalar(#1,#2){(#1\mid#2)}

\newcommand{\cb}{{\cal B}}
\newcommand{\cc}{{\cal C}}

\newcommand{\xbm}{(X,{\cal B},\mu)}

\newcommand{\ycn}{(Y,{\cal C},\nu)}
\newcommand{\ot}{\otimes}
\newcommand{\ov}{\overline}
\newcommand{\la}{\lambda}

\newcommand{\T}{{\mathbb{T}}}
\newcommand{\C}{{\mathbb{C}}}
\newcommand{\Z}{{\mathbb{Z}}}
\newcommand{\N}{{\mathbb{N}}}

\newcommand{\vep}{\varepsilon}
\newcommand{\va}{\varphi}

\newcommand{\mob}{\boldsymbol{\mu}}
\newcommand{\lio}{\boldsymbol{\lambda}}

\newcommand{\raz}{\mathbbm{1}}

\newcommand*\samethanks[1][\value{footnote}]{\footnotemark[#1]}

\title{Substitutions and M\"obius disjointness}
\author{S.\ Ferenczi \and J.\ Ku\l aga-Przymus\thanks{Research supported by Narodowe Centrum Nauki grant  UMO-2014/15/B/ST1/03736.} \and M.\ Lema\'nczyk\samethanks \and C.\ Mauduit}

\begin{document}
\bibliographystyle{siam}
\sloppy
\maketitle \normalsize

\thispagestyle{empty}

\begin{abstract}
We show that Sarnak's conjecture on M\"obius disjointness holds for all subshifts given by bijective substitutions and some other similar dynamical systems, e.g.\ those generated by Rudin-Shapiro type sequences.
\end{abstract}

\tableofcontents

\section{Introduction}

In 2010, Sarnak \cite{sarnak-lectures} formulated the following
conjecture: for each zero entropy topological dynamical system $(T,X)$ ($X$ is a compact metric space and $T$ is a homeomorphism of $X$), each $f\in C(X)$ and $x\in X$, we have
\beq\label{sarnakcon}
\frac1N\sum_{n\leq N}f(T^nx)\mob(n)\to 0,
\eeq
where $\mob\colon \N\to \C$ is the M\"obius function defined by $\mob(1)=1$, $\mob(p_1\cdot\ldots\cdot p_k)=(-1)^k$ for $k$ different prime numbers $p_i$, and $\mob(n)=0$ in the remaining cases. The conjecture has already been proved in numerous cases, e.g.\ \cite{Abdalaoui:2013rm,MR3121731,MR3043150,MR3095150,MR2986954,Do-Ka,Fe-Ma,
MR2981162,MR2877066,Kuaga-Przymus:2013fk,MR3347317,
MR2680394,MaRi2013,Ta1,Ve}.

The aim of the present paper is to show that Sarnak's conjecture holds for some classes of dynamical systems of number theoretic origin. Namely, for all dynamical systems given by bijective substitutions (a subclass of substitutions of constant length)~\cite{MR2590264} and also for other related systems given by some automata (e.g.\ by the sequences of the Rudin-Shapiro type).
Our approach is purely ergodic and the main tool is the theory of compact group extensions of rotations. Throughout, we deal with uniquely ergodic homeomorphisms, i.e.\ homeomorphisms possessing exactly one invariant measure (which has to be ergodic).

Given a bijective substitution $\theta$ over a finite alphabet $A$, we define its group cover substitution $\ov{\theta}$ over a subgroup $G$ of permutations of $A$, which hence carries an additional natural group structure. Since the dynamical system $(S,X(\theta))$\footnote{Here and in what follows, $S$ stands for the left shift on a closed $S$-invariant subset  of the space of two-sided sequences, i.e.\ $(S,X(\theta))$ is an example of a \emph{subshift}.} given by $\theta$ is a topological factor of the dynamical system $(S,X(\ov{\theta}))$ given by $\ov{\theta}$, it suffices to show that Sarnak's conjecture holds for $(S,X(\ov{\theta}))$. The group cover substitution $\ov{\theta}$ can be identified with a certain (generalized) Morse sequence $x$ and the associated dynamical system $(S,X(x))$ is a Morse system. This is where compact group extensions come into play -- each Morse dynamical system is (measure-theoretically) isomorphic to a compact group extension $(T_\psi,X\times G)$ given by a so called Morse cocycle $\psi\colon X\to G$ over a rotation~$T$, more precisely over an odometer $(T,X)$.\footnote{This and other relations between the dynamical systems described in this paragraph are illustrated in Figure~\ref{fig1}.} The main difficulty is that such compact group extensions have been studied so far mostly from the measure-theoretic point of view~\cite{MR0239047,MR937955, MR2590264} (the dynamical systems under considerations are uniquely ergodic). The underlying reason and, at the same time, the main obstacle for us is that Morse cocycles are in general not continuous. Thus, we cannot deal directly with such models -- Sarnak's conjecture requires topological systems. In order to bypass this difficulty, more tools are used. The Morse dynamical system $(S,X(x))$ turns out to have a Toeplitz dynamical system $(S,X(\widehat{x}))$ as a topological factor, which, in turn, is an almost {1-1}-extension of the odometer $(T,X)$. Moreover, the method of Toeplitz extensions~\cite{MR937955} allows us to find a dynamical system topologically isomorphic to $(S,X(x))$, which also has a form of a compact group extension: $(S_\varphi,X(\widehat{x})\times G)$ given by a continuous cocycle $\varphi\colon X(\widehat{x})\to G$. If we denote the (natural) factoring map from $(S,X(\widehat{x}))$ to $(T,X)$ by $p$, we have the following relation between the two cocycles: $\varphi=\psi\circ p$. Our goal will be to prove that Sarnak's conjecture holds for $(S_\varphi,X(\widehat{x})\times G)$.

The passage to group substitutions seems to be unavoidable for our method. Indeed, we require that the  substitution subshift has a  topological factor, determined by a Toeplitz sequence, which is an almost 1-1 extension of the maximal equicontinuous factor (the underlying odometer) and is measure-theoretically isomorphic to it. However, in \cite{MR3167382}, Section 4.4, it is proved that a bijective substitution need not have a symbolic factor which is measure-theoretically isomorphic to the maximal equicontinuous factor (this anwers a question  raised by Baake). For example, this surprising property holds for the substitution $a\mapsto aabaa$, $b\mapsto bcabb$ and $c\mapsto cbccc$ \cite{MR3167382}.

\begin{figure}[h!]
\begin{center}
	\begin{tikzpicture}[every text node part/.style={align=center}]
		\node (T) at (0,5.2) {{\bf continuous} cpt.\\ group extension\\$(S_\varphi,X(\widehat{x})\times G)$};
		\node (S) at (4,7) {Morse system\\ given by a Morse cocycle\\ (cpt. group extension) \\$(T_\psi,X\times G)$};
		\node (M) at (0,3.5) {Toeplitz system\\$(S,X(\widehat{x}))$};
		\node (B) at (0,1.9) {odometer\\ $(T,X)$};
		\node (TL) at (-4,7) {Morse system\\ given by $x$ \\ $(S,X(x))$} ;
		\draw[->] (S) edge[bend left=30] node[auto] {} (B);
		\draw[<-,dashed] (0.1,3) edge node[auto] {$\simeq$} (0.1,2.4);
		\draw[->] (M)+(-0.1,-0.5) -- (-0.1,2.4);
		\draw[->] (T) edge node[auto] {} (M);		
		\draw[->,dashed] (T) edge node[auto] {$\simeq$} (2.7,6.2);
		\draw[->] (TL) edge[bend left=-30] node[auto] {} (M);	
		\draw[->] (TL) edge node[auto] {$\simeq$} (T);	
		\draw[->,dashed] (TL) edge node[auto] {$\simeq$} (S);
	\end{tikzpicture}
\end{center}
\caption{Morse and Toeplitz dynamical systems on a diagram. Plain and dashed lines denote topological and measure-theoretical maps, respectively (all depicted systems are uniquely ergodic).}
\label{fig1}
\end{figure}
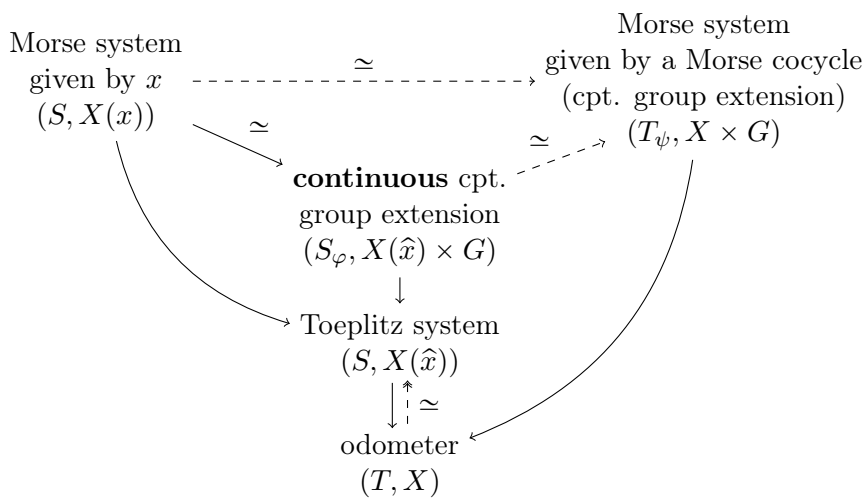

The first tool we use to deal with the continuous compact group extensions is the Katai-Bourgain-Sarnak-Ziegler criterion:

\begin{Th}[\cite{MR836415,MR2986954}, see also \cite{Ha}] \label{kbsz}
Assume that $(a_n)\subset\C$ is bounded and suppose that
$$
\frac1N\sum_{n\leq N} a_{nr}\ov{a}_{ns}\to0$$
for all sufficiently large different prime numbers $r,s$. Then
$$
\frac1N\sum_{n\leq N}a_n\lio(n)\to 0$$
for each  multiplicative\footnote{$\lio\colon\N\to\C$ is called {\em multiplicative} if $\lio(m\cdot n)=\lio(m)\cdot\lio(n)$ whenever $m,n$ are coprime. It is called \emph{aperiodic} whenever $\frac{1}{N}\sum_{n\leq N}\lio(an+b)\to 0$ for all $a,b\in\N$. The M\"obius function $\mob$ is multiplicative and aperiodic.} function $\lio\colon\N\to\C$, $|\lio|\leq1$.
\end{Th}
This criterion is applied to sequences of the form $a_n=f(T^nx)$, $n\geq1$, and, if satisfied, it yields a certain form of disjointness of different prime powers of the homeomorphism~$T$; notice that ,,sufficiently large'' in \cref{kbsz} may depend on $f$ and~$x$.

Our second tool is based on the method of lifting generic points in the Cartesian products of different prime powers $T^r$ and $T^s$ to almost {1-1}-extensions. It has already appeared in \cite{Ku-Le1}, where the almost {1-1}-extensions are chosen in such a way that the original cocycle considered on the extended space becomes continuous. Moreover, we will study ergodic joinings of $(T_\psi)^r$ and $(T_\psi)^s$ and show that in our case this set consists only of the relatively independent extensions of isomorphisms between $T^r$ and $T^s$ for different sufficiently large primes $r,s$. This will allow us to control generic points in the Cartesian product of the continuous  compact group extension $(S_\varphi)^r \times (S_\varphi)^s$.

In \cref{se5}, we give the main application of our method -- we prove Sarnak's conjecture for the dynamical systems given by:
\begin{itemize}
\item
all bijective substitutions (for a relationship of our result with a recent result of Drmota \cite{MR3287361}, see \cref{rem57}),
\item
certain subclass of regular (generalized) Morse sequences,
\item
certain sequences of the Rudin-Shapiro type.
\end{itemize}

In \cref{se6}, we compare our results concerning generalized Morse sequences and the Rudin-Shapiro type sequences with some earlier results in which~\eqref{sarnakcon} has been proved only for $f(y)=(-1)^{y[0]}$. Using spectral approach, we prove, in some cases, that the validity of~\eqref{sarnakcon} for such $f$ yields Sarnak's conjecture in its full form for the corresponding dynamical system. In particular, we show that Sarnak's conjecture holds for the dynamical systems given by Kakutani sequences~\cite{MR3043150,MR2981162}.\footnote{This result does not seem to follow by~\cref{se4}.} Sometimes, however, it seems that more than one function satisfying~\eqref{sarnakcon} is necessary for the validity of Sarnak's conjecture. E.g.,\ this seems to be the case for the dynamical systems given by the Rudin-Shapiro type sequences. Here, \eqref{sarnakcon} for $f$ has been been proved in~\cite{MaRi2013,Ta1} by a purely combinatorial approach. The methods developed in~\cite{MaRi2013} seem to be flexible enough to give~\eqref{sarnakcon} for finitely many functions described in~\cref{aplikacje}, hence yields one more proof of Sarnak's conjecture for the corresponding dynamical system.

In \cref{ostatnia1}, we compare our results with a recent work of Veech~\cite{Ve}. He provides a proof of Sarnak's conjecture for a class of dynamical systems given by some sequences over an arbitrary compact group. These sequences turn out to be generalizations of aforementioned Kakutani sequences. In particular, \cite{Ve} gives an alternative proof of Sarnak's conjecture for Kakutani systems first proved in \cite{MR3043150,MR2981162}.

\section{Basic tools}

\subsection{Spectral theory}

For an ergodic automorphism $T$ of a standard probability Borel space $(X,\mathcal{B},\mu)$, we consider the associated Koopman operator on $L^2(X,\mathcal{B},\mu)$
 given by $U_T(f)=f\circ T$. Then there exist elements $f_n \in L^2(X,\mathcal{B},\mu)$, $n\geq 1$,  such that
\begin{equation}\label{spect}
L^2(X,\mathcal{B},\mu)=\bigoplus_{n\geq 1}\mathbb{Z}(f_n) \text{ and }\sigma_{f_1} \gg \sigma_{f_2} \gg \dots,
\end{equation}
where $\mathbb{Z}(f)=\ov{\text{span}}\{U_T^n (f) : n\in\Z\}$ is the \emph{cyclic space} generated by $f$ and $\sigma_f$ denotes the only finite positive Borel measure on $\T$ such that $\int_X f\circ T^n \cdot\ov{f} \ d\mu=\int_\T z^n\ d \sigma_f(z)$ for each $n\in\Z$ ($\sigma_f$ is called the \emph{spectral measure of $f$}). The class of all measures equivalent to $\sigma_{f_1}$ in the above decomposition is called \emph{the maximal spectral type} of $U_T$ and \eqref{spect} is called a \emph{spectral decomposition}. We say that the maximal spectral type is realized by $f\in L^2(X,\mathcal{B},\mu)$ if $\sigma_{f}$ is equivalent to $\sigma_{f_1}$.

If $L^2(X,\mathcal{B},\mu)=\Z(f_1)\oplus \dots \oplus \Z(f_k)$ for some $f_i\in L^2(X,\mathcal{B},\mu)$, we say that $U_T$ has multiplicity at most $k$. If no such $k\geq 1$ exists, the multiplicity of $U_T$ is infinite. If $k=1$, $U_T$ is said to have \emph{simple spectrum}.

Recall that $T$ has \emph{discrete spectrum} if the maximal spectral type of $U_T$ is purely discrete. Equivalently, $L^2\xbm$ is generated by the eigenfunctions of $U_T$. By the Halmos-von Neumann theorem, $T$ is, up to isomorphism, an ergodic rotation on a compact monothetic group. If, additionally, all eigenvalues of $U_T$ are roots of unity, $T$ is said to have {\em rational discrete spectrum.}

For more information on the spectral theory see, e.g.,~\cite{MR2140546}.

\subsection{Joinings}
Recall that if $T$ and $S$ are ergodic automorphisms on $\xbm$ and $\ycn$, respectively, then by a {\em joining} between $T$ and $S$ we mean any $T\times S$-invariant measure $\kappa$ on  $(X\times Y,\cb\ot\cc)$ whose projections on $X$ and $Y$ are $\mu$ and $\nu$, respectively. We denote by $J(T,S)$ the set of joinings between $T$ and $S$ and by $J^e(T,S)$ the subset of ergodic joinings. Clearly, $\mu\otimes\nu \in J(T,S)$. If $T$ and $S$ are isomorphic, an isomorphism given by $R\colon\xbm\to\ycn$, then the measure $\mu_R$ determined by $\mu_R(B\times C)=\mu(B\cap R^{-1}C)$, $B\in\cb,C\in\cc$, belongs to $J^e(T,S)$. It is concentrated on the graph of $R$ and is called a {\em graph} joining. When $T=S$, we speak about self-joinings of $T$ and each graph self-joining is given by some element from the centralizer $C(T)$ of $T$.\footnote{The centralizer $C(T)$ consists of automorphisms of $\xbm$ commuting with $T$.}

Suppose that $T$ and $S$ are isomorphic, where the isomorphism is given by $R\colon X\to Y$, and have extensions to $\ov{T}$ on $(\ov{X},\ov{\mathcal{B}},\ov{\mu})$  and  $\ov{S}$ on $(\ov{Y},\ov{\mathcal{C}},\ov{\nu})$, respectively. The \emph{relatively independent extension} of $\mu_R$ (to a joining of $\ov{T}$ and $\ov{S}$) is denoted by $\widetilde{\mu}_R$ and determined by
$$
\int_{\ov{X}\times \ov{Y}} F\otimes G\ d\widetilde{\mu}_R=\int_{X} \mathbb{E}(F|X) \cdot \mathbb{E}(G|Y)\circ R \ d\mu
$$
for $F\in L^2(\ov{X},\ov{\mu})$, $G\in L^2(\ov{Y},\ov{\nu})$.

\subsection{Compact group extensions}
Assume that $T$ is an ergodic automorphism of a standard Borel probability space $\xbm$. Let $G$ be a compact metric group with Haar measure $m_G$.
\begin{Def}
Any measurable map $\psi\colon X\to G$ is called a {\em cocycle}. The automorphism $T_\psi$ of $(X\times G,\cb\ot\cb(G),\mu\ot m_G)$ defined by
$$
T_\psi(x,g):=(Tx,\psi(x)g)$$
is called a $G$-{\em extension} of $T$ (it is an example of a {\em compact group extension of} $T$).  We say that $\psi$ is {\em ergodic} if $T_\psi$ is ergodic.
\end{Def}

Compact group extensions enjoy the following relative unique ergodicity property.

\begin{Lemma}[\cite{MR603625}]\label{rue} If $T_\psi$ is ergodic (i.e.\ if the product measure $\mu\ot m_G$ is ergodic) then $\mu\ot m_G$ is the only $T_\psi$-invariant measure projecting onto $\mu$.\end{Lemma}

Let $\tau_g$ be an automorphism of $(X\times G,\mu\ot m_G)$ given by
$$
\tau_g(x,g')=(x,g'\cdot g)\;\;\mbox{for each}\;\;g'\in G.$$
Then $T_\psi\circ \tau_g=\tau_g\circ T_\psi$, that is, $\tau_g$ is an element of the centralizer $C(T_\psi)$ of $T_\psi$.

\begin{Prop}[\cite{MR1141362}] \label{centrm} Assume that $T$ is ergodic and $\psi\colon X\to G$ is ergodic as well. Assume additionally that $T$ has discrete spectrum. Then each $\widetilde{S}\in C(T_\psi)$  is a lift of some $S\in C(T)$. More precisely, $\widetilde{S}=S_{f,v}$, where $S_{f,v}(x,g)=(Sx,f(x)v(g))$ for some $S\in C(T)$, some measurable $f\colon X\to G$ and some continuous group automorphism $v\colon G\to G$. Moreover, if $\widetilde{S}$ and $\ov{S}$ are two lifts of $S\in C(T)$ then $\widetilde{S}=\ov{S}\circ \tau_{g_0}$ for some $g_0\in G$.\end{Prop}

\begin{Def}
We will say that $T_\psi$ has $G$-{\em trivial centralizer} if
$$
C(T_\psi)=\{T_\psi^k\circ\tau_g :  k\in\Z, \; g\in G\}.
$$
\end{Def}

\begin{Def}
Let $H\subset G$ be a closed subgroup. The corresponding factor-automorphism $T_{\psi H}$ of $(X\times G/H,\mu\ot m_{G/H})$ given by
$$
T_{\psi H}(x,gH):=(Tx,\psi(x)gH)
$$
is called a \emph{natural factor} of $T_\psi$. It is called \emph{nontrivial} if $H\neq G$, and it is called \emph{normal} whenever $H$ is normal.
\end{Def}

\begin{Remark}
Notice that a power of a group extension is clearly a group extension: $(T_\psi)^r=T^r_{\psi^{(r)}}$\footnote{$\psi^{(r)}(x):=
\psi(x)\psi(Tx)\ldots\psi(T^{r-1}x)$ for $r\geq0$ and extends to $r\in\Z$ so that the cocycle identity
 $\psi^{(m+n)}(x)=\psi^{(m)}(x)\psi^{(n)}(T^mx)$ holds for every $m,n\in\Z$.} and the passage to natural factors is ``commutative'':
 $((T_{\psi})^r)_H=(T_{\psi H})^r$.
\end{Remark}

We need some facts about joinings of compact group extensions.

\begin{Th}[\cite{MR1141362}]\label{mkm} Assume that $T$ is ergodic. Assume that $S\in C(T)$ and let $\psi_i:X\to G$ be an ergodic cocycle, $i=1,2$. Assume that $\kappa\in J^e(T_{\psi_1},T_{\psi_2})$ and projects on the graph self-joining $\mu_S$ of $T$. Then there are two closed normal subgroups $H_1,H_2\subset G$ and an isomorphism $\ov{S}$ (a lift of $S$) between the two normal natural factors
$T_{\psi_1 H_1}$ and $T_{\psi_2 H_2}$ such that
\[
\kappa=\widetilde{(m_{G/H_1})_{\ov{S}}},
\]
i.e.\ $\kappa$ is the relatively independent extension of the graph joining $(m_{G/H_1})_{\ov{S}}\in J^e(T_{\psi_1 H_1},T_{\psi_2 H_2})$ given by the isomorphism $\ov{S}$.
\end{Th}

\begin{Remark}\label{zastmkm}
Suppose that $T$ has rational discrete spectrum and $(T_\psi)^r$ and $(T_\psi)^s$ are ergodic. Then $T^r$ and $T^s$ are isomorphic and the only ergodic joinings between them are the graph joinings. By \cref{mkm}, if there is no isomorphism between nontrivial normal natural factors of $(T_\psi)^r$ and $(T_\psi)^{s}$, then there are no ergodic joinings between  $(T_\psi)^r$ and $(T_\psi)^s$, except for the ``most independent'' ones: the relatively independent extensions of isomorphisms between $T^r$ and $T^s$. Notice also that if such a relative product is ergodic then automatically, by \cref{rue}, it is the only invariant measure on $X\times G\times X\times G$ projecting on the graph of the isomorphism.
\end{Remark}

\subsection{Generic points}
Let $T$ be a homeomorphism of a compact metric space $X$. Let $\mu$ be a $T$-invariant Borel probability measure on $X$.
\begin{Def}
We say that $x\in X$ is \emph{generic} for $\mu$ if $\frac1N\sum_{n\leq N}\delta_{T^nx}\to \mu$ weakly. If the convergence to $\mu$ takes place only along a subsequence $(N_k)$ then $x$ is called \emph{quasi-generic} for $\mu$.
\end{Def}
\begin{Remark}
Notice that, by the compactness of $X$, the space of probability measures on $X$ is also compact, hence each point is quasi-generic for some $T$-invariant measure.
\end{Remark}

\section{Basic objects}

\subsection{Odometers, Morse cocycles and Toeplitz extensions}\label{se:morsecoc}

\paragraph{Odometers}
Assume that $(n_t)_{t\geq0}$ satisfies $n_0=1$ and $n_t\divides n_{t+1}$ with $\la_t:=n_{t+1}/n_t\geq2$ for $t\geq0$. Consider
$X:=\prod_{t\geq 0}\Z/\la_t\Z$ with the product topology and the group law given by addition mod $\la_t$, with carrying the remainder to the right. This makes $X$ a compact metric Abelian group.
We define the translation $T$  by $(1,0,0,\ldots)$:
$$T(x_0,x_1,x_2,\ldots)=(x_0+1,x_1,x_2,\ldots)$$
to obtain $(X,\cb,m_{X},T)$ -- an ergodic rotation.
\begin{Def}
$T$ defined above is called an \emph{odometer}.
\end{Def}

\begin{Remark}
Odometer $T$ defined above has rational discrete spectrum given by the $n_t$-roots of unity, $t\geq0$. For each $t\geq0$, there is a Rokhlin tower ${\cal D}^t:=\{D^t_0,D_1^t,\ldots, D^t_{n_t-1}\}$, i.e.\ a partition of $X$ for which  $T^{i}D^t_0=D^t_{i\;{\rm mod}\;n_t}$ for each $i\geq 0$ (by ergodicity, such a tower is unique up to cyclic permutation of the levels). Indeed, $D_0^0=X$ and we set
$$
D^t_0:=\{x\in X :  x_0=\ldots=x_{t-1}=0\},\ t\geq 1.
$$
Clearly, the partition ${\cal D}^{t+1}$ is finer that ${\cal D}^t$ and the sequence of such partitions tends to the partition into points.
\end{Remark}
\begin{Remark}[cf.\ \cref{zastmkm}]
Notice that for each $r\geq1$,
\beq\label{isopowers}
\mbox{$T^r$ is isomorphic to $T$ whenever $T^r$ is ergodic.}\eeq
Indeed, $T^r$ has the same spectrum as $T$. To see the isomorphism more directly, notice that ${\rm gcd}(r,n_t)=1$ and $T^r$ permutes the levels of ${\cal D}^t$ -- this extends to an isomorphism map between $T$ and $T^r$.
\end{Remark}
\begin{Remark}[cf.\ \cref{zastmkm}]
Since $T$ has discrete spectrum, its only ergodic joinings are graph measures ${(m_X)}_W$, where $W\in C(T)$ is another rotation on $X$~\cite{MR1958753}.
It easily follows that
\begin{equation}\label{genodo}
  \parbox{0.8\linewidth}{each point $(x,y)\in X\times X$ is generic for an ergodic self-joining of the form ${(m_X)}_W$.}
\end{equation}
Indeed, define $W$ as the translation by $x-y$.
\end{Remark}

\paragraph{Morse cocycles}
Assume that $G$ is a compact metric group and $(T,X)$ is an odometer.

\begin{Def}[\cite{MR1685402,MR937955}]
We say that $\psi\colon X\to G$ is a \emph{Morse cocycle} if $\psi$ is constant on each $D^t_i$, $t\geq0$, $i=0,1,\ldots, n_t-2$ ($\psi|_{D^t_i}$ may depend on $i$).
\end{Def}
\begin{Remark}
To see  what are the values of a Morse cocycle $\psi$ on $D^{t}_{n_t-1}$, we first pass to the levels $D^{t+1}_{jn_t-1}$ for $j=1,\ldots,\la_{t+1}-1$, and read the values $\psi|_{D^{t+1}_{jn_t-1}}$. To read the values on
$D^{t+1}_{n_{t+1}-1}$ ($n_{t+1}=\la_{t+1}n_t$), we pass to ${\cal D}^{t+2}$ etc. It is clear that  $\psi$ defined in this way is continuous everywhere (as the levels of the towers are clopen sets) except perhaps  one point (given by the intersection of the top levels of all towers). Notice also that whenever $G$ is finite then a Morse cocycle cannot be continuous unless it is constant on {\bf each} level
of the tower~${\cal D}^{t_0}$ for some $t_0$. In this case, $T_\psi$, if ergodic, is a direct product of $T$ with a rotation on $G$. In particular, Sarnak's conjecture holds for $T_\psi$.
\end{Remark}

\begin{Remark}
The class of group extensions given by Morse cocycles is (up to measure-theoretic isomorphism) the same as the class of dynamical systems generated by generalized Morse sequences, see \cite{MR1685402,MR0239047,MR937955,MR970577},  which we consider in the next section.
\end{Remark}
\paragraph{Toeplitz extensions}
Morse cocycles yield extensions of odometers which are special cases of so called {\em Toeplitz extensions} studied in \cite{MR937955} (cf.\ earlier \cite{MR0240056,MR556675} and the last section). Toeplitz extensions are also given by cocycles over odometers but in the definition of such cocycles we are letting more than one level have non-constant values, as in the example below.\footnote{It is however required that the numbers of levels of ${\cal D}^t$ on which the cocycle is non-constant divided by $n_t$ goes to zero. A reason for that is that we want to obtain a regular Toeplitz sequence which is behind such a construction, see \cite{MR937955} for more details.}

\begin{Example}\label{exampl}
Let $\la_t:=2$ for each $t\geq0$ and $G:=\Z/2\Z$. We define $\psi\colon X\to\Z/2\Z$ so that at stage $t$ it is defined on each $D^t_i$, except for $i=2^{t-1}-1$ and $i=2^{t}-1$. Then,  when we pass to ${\cal D}^{t+1}$, on the levels $D^{t+1}_{2^{t-1}-1}$ and $D^{t+1}_{2^{t}+2^{t-1}-1}$ ($\psi$ must be defined here at this stage of the construction), we set the values 0 and 1 (or 1 and 0), respectively.
\end{Example}
The class of Toeplitz extensions of the dyadic odometer described in \cref{exampl} was considered in \cite{MR937955,MR749448}. The dynamical systems corresponding to the Rudin-Shapiro type  sequences (see Section III.2 in \cite{MR937955}) are  in this class.

\subsection{Generalized Morse sequences}\label{mrsna}
Let $G$ be a compact metric group with the unit $e$.
\begin{Def}
Let $b^t\in G^{\la_t}$ be a block over $G$ of length $|b^t|=\la_t\geq 2$ and $b^t[0]=e$, $t\geq0$.  The associated \emph{(generalized) Morse sequence} is defined by
\beq\label{mo1}
x:=b^{0}\times b^{1}\times\ldots,\eeq
where $B\times C=(B\circ c[0])(B\circ c[1])\ldots (B\circ c{[|C|-1]})$ and $B\circ g:=(b_0g,\ldots,b_{|B|-1}g)$ for $B,C$ blocks over $G$ and $g\in G$. By $(S,X(x))$ we denote the  subshift corresponding to $x$ ($X(x)\subset G^{\Z}$).
\end{Def}

\begin{Example}
Generalized Morse sequences for $G=\Z/2\Z$ were first studied in~\cite{MR0239047}. If $b^t\in\{00,01\}$, $t\geq0$, we speak about \emph{Kakutani sequences}~\cite{MR638749}.
\end{Example}

\begin{Def}[\cite{MR0255766}]
We say that $u\in G^\N$ is a \emph{Toeplitz sequence} whenever for each $n\in \N$ there exists $k_n\geq 1$ such that $u$ is constant on the arithmetic progression $n+k_n\N$.\footnote{For the theory of dynamical systems given by Toeplitz sequences, see e.g.\ \cite{MR2180227}.}
\end{Def}

\begin{Lemma}[cf.\ Figure~\ref{fig1}]\label{lmo1} Let $x=b^0\times b^1\times\ldots$ be a Morse sequence. The map
\beq\label{mo2}
y\mapsto \widehat{y},\;\;\widehat{y}[n]:=y[n+1]y[n]^{-1}
\eeq yields an equivariant map between $(S,X(x))$ and $(S,X(\widehat{x}))$. Moreover, $\widehat{x}$ is a Toeplitz sequence.
\end{Lemma}
\begin{proof}
The first part is obvious. For the second, notice that for each $t\geq1$, we have
$$
x=c_t\times z_t, \text{ where }c_t=b^{0}\times\ldots\times b^{t-1},\; z_t=b^t\times b^{t+1}\times\ldots,
$$
whence $x$ is a concatenation of blocks of the form $c_t\circ g$. Moreover,
$$
c_t[n+1]c_t[n]^{-1}=(c_t\circ g)[n+1](c_t\circ g)[n]^{-1}\text{ for }n=0,\ldots,|c_t|-2.
$$
It follows that
\begin{equation}\label{xdach}
\widehat{x}=\widehat{c}_t\ast\widehat{c}_t\ast\widehat{c}_t\ast\ldots,
\end{equation}
where ``$\ast$'' stands for the unfilled place of $\widehat{x}$ at the stage $t\geq1$.
\end{proof}
\begin{Remark}
Toeplitz sequence $\widehat{x}$ from~\eqref{xdach} is regular in the sense of~\cite{MR0255766}. Hence $(S,X(\widehat{x}))$ is uniquely ergodic~\cite{MR2180227}.
\end{Remark}

\begin{Lemma}[cf.\ Figure~\ref{fig1}]\label{lmo2}
$(S,X(x))$ is topologically isomorphic to $(S_{\va},X(\widehat{x})\times G)$, where $\va\colon X(\widehat{x})\to G$ is the continuous cocycle given by
$$
\va(z)=z[0], \text{ for each }z\in X(\widehat{x}).$$
\end{Lemma}
\begin{proof}
The topological isomorphism is given by the (equivariant) map
$$
y\mapsto (\widehat{y},y[0]).$$
Indeed, this map is continuous as $\va$ is continuous, it is onto since
$\widehat{y}=\widehat{y\circ g}$, and finally it is 1-1 since $y[n+1]$ is determined by $y[n]$ and $\widehat{y}[n]$.
\end{proof}

\begin{Remark}[cf.\ Figure~\ref{fig1}]\label{uw16}
Let $x$ be the Morse sequence given by~\eqref{mo1} with $n_t:=|c_t|=|b^0 \times \dots \times b^{t-1}|$, $t\geq 1$. Then the Toeplitz system $(S,X(\widehat{x}))$ has the $\{n_t\}$-odometer $(T,X)$ as its topological factor. Moreover, the Morse dynamical system $(S,X(x))$ is given by a Morse cocycle $\psi$ over $T$. The values $(\psi|_{D_0^t},\dots,\psi|_{D_{n_t-2}^t})$ are determined by $\widehat{c}_t$:
$$
\psi|_{D^t_i}=\widehat{c}_t[i]\text{ for } 0\leq i\leq n_t-2,  t\geq 1.
$$
It follows that
$$
b^0[0]=e,\ b^0[i]=\prod_{j=0}^{i-1}\widehat{c}_1[j],1\leq i\leq \lambda_0-1
$$
and then, inductively ($c_t=b^0\times \ldots\times b^{t-1}$),
$$
b^{t+1}[0]=e,\ b^{t+1}[i]=\prod_{j=i}^1 (\widehat{c}_{t+1}[jn_t-1] c_t[n_t-1]), 1\leq i\leq \lambda_{t+1}-1
$$
(for more details see, e.g.\ \cite{MR937955,MR1758324}). However, the Morse cocycle is not continuous. The passage to the Toeplitz dynamical system from \cref{lmo2} allows us to get its continued version.
\end{Remark}

\begin{Lemma}\label{lmo3}
The normal natural factors of Morse dynamical systems over $G$ are  Morse dynamical systems over $G/H$.
\end{Lemma}

\begin{proof}
The assertion follows immediately from the equality $B\times C\bmod H=(B\bmod H)\times (C\bmod H)$ and \cref{uw16}.
\end{proof}

\subsection{Bijective substitutions}\label{bsgcnf}
Fix a finite alphabet $A$ with $|A|=r\geq 2$.

\begin{Def}
A map $\theta\colon A\to A^\lambda$ ($\lambda\geq 1$) is called a {\em substitution on} $A$ {\em of constant length} $\la$ (in what follows, simply a {\em substitution}) if there exists $n\geq 1$ such that for each $a,a'\in A$ there exists $k$ satisfying $\theta^n(a)[k]=a'$. We extend $\theta$ first to a map on blocks over $A$, then to a map $\theta\colon A^{\N}\to A^{\N}$. We will assume that $\theta(a_0)[0]=a_0$. By iterating $\theta$ at $a_0$, we obtain a fixed point for the map $\theta\colon A^{\N}\to A^{\N}$ and we denote by
$(S,X(\theta))$ the corresponding subshift of $A^\Z$.
\end{Def}
\begin{Remark}[see Chapter 5 in~\cite{MR2590264}]
For each substitution $\theta$, $(S,X(\theta))$ is strictly ergodic.
\end{Remark}

\begin{Remark}\label{f13}
Let $\theta\colon A\to A^\lambda$ be a substitution such that $(S,X(\theta))$ is aperiodic. Recall that then for each $y\in X(\theta)$ there is a unique sequence $(i_t(y))_{t\geq1}\subset \Z$ (\emph{$t$-skeleton}) with $i_t(y)\in[-\la^t+1,0]$ such that $y[i_t+k\la^t,i_t+(k+1)\la^t-1]=\theta^t(a_{k,t})$ for each $k\in\Z$ and some letters $a_{k,t}\in A$. This allows us to define the corresponding towers of height $\la^t$ by setting the base of the $t$-tower
 $$
 D^t_0:=\{y\in X(\theta):i_t(y)=0\}
 $$
 to obtain $\bigcup_{i=0}^{\la^t-1}S^iD^t_0=X(\theta)$.
\end{Remark}
\begin{Def}[\cite{Mosse:1992}]
We say that substitution $\theta$ is \emph{recognizable} if there exists a constant $M>0$ such that if $y\in X(\theta)$, $t\geq1$ and $i\in[-\la^t+1,0]$ satisfy
$$ y[i,i+M\la^t-1]=\theta^t(b_1)\ldots\theta^t(b_M)$$ for some $b_1,\ldots,b_M\in A$ then $i=i_t(y)$. We say that $M$ is a \emph{constant of recognizability}.
\end{Def}
\begin{Remark}[\cite{Mosse:1992,MR2590264}]
Each substitution $\theta$ such that $(S,X(\theta))$ is aperiodic, is recognizable. In what follows, we will tacitly assume that we deal with recognizable substitutions.
\end{Remark}

\begin{Remark}\label{f13a}
Suppose that $\theta$ is recognizable. It follows that each function $\raz_{D^t_0}$ depends on not more than $M\la^t$ coordinates.
\end{Remark}

\begin{Def}[\cite{MR2590264}]
Substitution $\theta$ is called {\em bijective} if
$$
\ov{d}(\theta(a),\theta(a')):=\frac{|\{0\leq k\leq \lambda-1 : \theta(a)[k]\neq\theta(a')[k]\}|}{\lambda}=1,\text{ whenever }a\neq a'
$$
or, equivalently, the maps $\sigma_i(a):=\theta(a)[i]$ are bijections of $A$, $i=0,\ldots,\la-1$.
\end{Def}
\begin{Remark}
We can assume (wlog) that $\sigma_0=Id$ by considering, if necessary, its power.
\end{Remark}

\begin{Def}
Let $G$ be a finite group with the unit $e$. A substitution $\theta\colon G\to G^\lambda$ is called a \emph{group substitution}  whenever
$$
\theta(g)=\theta(e)\circ g \text{ for each }g\in G.
$$
\end{Def}

\begin{Remark}\label{uw:22}
Each group substitution is bijective. Moreover, each group substitution can be identified with the Morse sequence $\theta(e) \times \theta(e) \times \dots$
\end{Remark}

\begin{Lemma}[cf.\ \cref{lmo3} and \cref{uw:22}]\label{lmo3a}
The normal natural factors of dynamical systems given by group substitutions are determined by group substitutions.
\end{Lemma}
\begin{proof}
Consider the group substitution given by $e\mapsto B$ (i.e.\ $g\mapsto B\circ g$). Then
$$
gH\mapsto B\circ g\;{\rm mod}\;H
$$
yields a bijective substitution as in each column of the matrix corresponding to the group substitution we see all elements of $G$; in particular,  by taking them mod $H$, we see all elements of $G/H$.
\end{proof}

Denote by $\mathscr{S}_r$ the group of permutations of $A$.
Define $\widetilde{\theta}\colon \mathscr{S}_r\to \mathscr{S}_r^{\la}$ by setting
\beq\label{cover}
\widetilde{\theta}(\tau)=(\sigma_0\circ\tau,\sigma_1\circ\tau,
\ldots,\sigma_{\la-1}\circ\tau)=\widetilde{\theta}(Id)\circ\tau\eeq
for each $\tau\in\mathscr{S}_r$. Let  $G\subset \mathscr{S}_r$ be the subgroup generated by $\sigma_0,\sigma_1,\ldots,\sigma_{\la-1}$ and define
\beq\label{cover1}
\ov{\theta}(\tau):=\widetilde{\theta}(\tau)\;\mbox{for}\;\tau\in G.
\eeq
\begin{Def}[cf.\ \cref{lgc1} below]
We call $\ov{\theta}$ the {\em group cover substitution} of $\theta$.
\end{Def}
\begin{Lemma}
$\ov{\theta}$ is a (bijective) substitution.
\end{Lemma}
\begin{proof}
Notice that if $\ov{\theta}^n(\sigma_0)[j]=\tau$ then, in $\ov{\theta}^{n+1}(\sigma_0)$, we can find the block $(\sigma_0\circ\tau,\ldots,\sigma_{\la-1}\circ\tau)$. Since all elements in $G$ are of finite order ($\sigma_i^{-1}=\sigma_i^{r!-1}$), it follows by induction that,  for some $n$, we will see all symbols from $G$ on $\ov{\theta}^n(\sigma_0)$.
\end{proof}

\begin{Lemma}\label{lgc1}
$(S,X(\theta))$ is a topological factor of its group cover substitution $(S,X(\ov{\theta}))$.\end{Lemma}
\begin{proof}
 We define an equivalence relation on $G$  by setting $\tau\equiv\tau'$ if $\tau(0)=\tau'(0)$.\footnote{This relation is  called \emph{$\theta$-consistent}, see \cite{MR0293611,MR985556}.} For $y\in X(\ov{\theta})$, set
$$
F(y)[n]:=(y[n])(0).
$$
Notice that the image of $F$ equals $X(\theta)$, $F$ is equivariant and takes the same values on the equivalence classes of $\equiv$. Finally, notice that $\{\tau(0): \tau\in G\}=A$ since $\theta$ is a substitution, whence $G$ acts transitively on $A$.
\end{proof}

\begin{Remark}[cf.\ \cref{uw:22}]\label{prz1}
The group cover substitution $\ov{\theta}$ can be identified with the Morse sequence $B\times B\times\ldots$ (over $G$), where $B=(\sigma_0,\sigma_1,\ldots,\sigma_{\la-1})$.
\end{Remark}

\begin{Remark} \label{rmo}
Notice that, in order to prove Sarnak's conjecture for a bijective substitution $\theta$, it suffices to prove it for $(S_{\va},X(\widehat{x})\times G)$, where $x=B\times B\times\ldots$ as $(S,X(\theta))$  is its topological factor (see \cref{lgc1} and \cref{prz1}). Notice also that we do not claim that for $(S,X(\theta))$ the Toeplitz dynamical system $(S,X(\widehat{x}))$ is its topological factor (even though the odometer is its topological factor). In fact, there is a counterexample to such a claim due to Herning~\cite{MR3167382}.
\end{Remark}

\section{Sarnak's conjecture for finite group extensions}\label{se4}

\subsection{Lifting generic points for compact group extensions}
We now recall a basic result on lifting generic points from \cite{Ku-Le1}.
Assume that $\ov{T}_i$ ($T_i$) is a uniquely ergodic homeomorphism, with a unique invariant measure $\ov{\mu}_i$ ($\mu_i$),  of a compact metric space $\ov{X}_i$ ($X_i$), $i=1,2$. Assume, moreover, that $\pi_i\colon\ov{X}_i\to X_i$
is continuous and yields $(T_i,X_i)$ a topological factor of $(\ov{T}_i,\ov{X}_i)$.

\begin{Prop}[\cite{Ku-Le1}]\label{dl1}
Assume that $(\ov{T}_i,\ov{X}_i,\ov{\mu}_i)$ and $(T_i,X_i,\mu_i)$ are measure-theoretically isomorphic. Assume, moreover that $(T_i,X_i,\mu_i)$ is measure-theoretically coalescent\footnote{An automorphism $T$ of $\xbm$ is called {\em coalescent} \cite{MR0230877} if each endomorphism commuting with $T$ is invertible.}  for $i=1,2$. Assume that $(x_1,x_2)\in X_1\times X_2$ is generic for an ergodic $T_1\times T_2$-invariant measure $\rho$. Then there exists a unique  $\ov{T}_1\times \ov{T}_2$-invariant measure $\ov{\rho}$, such that  each pair $(\ov{x}_1,\ov{x}_2)\in (\pi_1\times\pi_2)^{-1}(x_1,x_2)$ is generic for $\ov{\rho}$. Moreover,  the system $(\ov{T}_1\times\ov{T}_2,\ov{\rho})$ is isomorphic to $(T_1\times T_2,\rho)$.
\end{Prop}

Let $T$ be an odometer acting on $(X,\mathcal{B},m_X)$ and let $\ov{T}$ be a uniquely ergodic homeomorphism of $\ov{X}$ (with the unique invariant measure $\ov{m_X}$) such that $\pi\colon\ov{X}\to X$ is a topological factor map, and $(T,m_X)$ and $(\ov{T},\ov{m_X})$ are measure-theoretically isomorphic (then $\pi$ is a.e.\ 1-1 as transformations with discrete spectrum are coalescent).  Assume that $\psi\colon X\to G$ is ergodic and such that the cocycle $\ov{\psi}\colon \ov{X}\to G$ given by
\beq\label{uciagl}
\ov{\psi}(\ov{x}):=\psi(\pi(\ov{x}))\;\;\mbox{is continuous}\eeq
 (it is automatically ergodic as $(T_\psi, X\times G,m_X\ot m_G)$ and $(\ov{T}_{\ov{\psi}},\ov{X}\times G,\ov{m_X}\ot m_G)$ are isomorphic).
Assume that $r\neq s$ are such that $(T_\psi)^r$ and $(T_\psi)^s$ are ergodic, hence $T^r$ and $T^s$ are isomorphic (and they are isomorphic to $T$).

\begin{Prop}\label{narz}
Assume that the only ergodic joinings between $(T_\psi)^r$ and $(T_\psi)^s$ are the relatively independent extensions over the graphs of isomorphisms between $T^r$ and $T^s$. Let $\ov{x}\in \ov{X}$ and let $\rho={(m_X)}_R$ be the (ergodic) graph joining for which the point $(\pi(\ov{x}),\pi(\ov{x}))$ is generic. Then for each $g\in G$, the point $((\ov{x},g),(\ov{x},g))$ is generic for the $\widetilde{\ov{\rho}}$, where $\ov{\rho}$ comes from \cref{dl1} ($\sim$~stands for the relatively independent extension). Moreover,
\beq\label{a10}
((\ov{T}_{\ov{\psi}})^r\times (\ov{T}_{\ov{\psi}})^s,\widetilde{\ov{\rho}})\;\;\mbox{and}\;\;
((T_\psi)^r\times (T_\psi)^s,\widetilde{{(m_X)}_R})\;\mbox{are isomorphic}.\eeq
\end{Prop}
\begin{proof} The point $((\ov{x},g),(\ov{x},g))$ is quasi-generic for a $(\ov{T}_{\ov{\psi}})^r\times (\ov{T}_{\ov{\psi}})^s$-invariant measure~$\kappa$. By~\cref{dl1}, $(\ov{x},\ov{x})$ is generic for $\ov{\rho}$. Therefore, the projection of $\kappa$ on $\ov{X}\times\ov{X}$ is equal to $\ov{\rho}$. Using \cref{rue} (applied to $\ov{T}^r\times\ov{T}^s$, $\ov{\psi}^{(r)}\times\ov{\psi}^{(s)}$ and $\ov{\rho}$), to conclude, we only need to prove that
$((\ov{T}_{\ov{\psi}})^r\times (\ov{T}_{\ov{\psi}})^s,\widetilde{\ov{\rho}})$ is ergodic. Notice that~\eqref{a10} is obvious since $\rho={(m_X)}_R$ and $\ov{\rho}$ yield isomorphic systems. This gives immediately that $((\ov{T}_{\ov{\psi}})^r\times (\ov{T}_{\ov{\psi}})^s,\widetilde{\ov{\rho}})$ is ergodic, whence $\kappa=\widetilde{\ov{\rho}}$.
\end{proof}

\subsection{Criterion for the validity of Sarnak's conjecture for finite group extensions}
In this section, we assume that $(T,X,\mathcal{B},\mu)$ is an ergodic transformation with discrete spectrum and $\psi\colon X\to G$ is an ergodic cocycle with values in a finite group~$G$.

\begin{Lemma} \label{centpot} Let $m=|G|$. Assume that $r\geq 2$ is an integer such that $(T_{\psi})^r$ is ergodic and ${\rm gcd}(r,m)=1$. Then $C(T_\psi)=C((T_\psi)^r)$.
\end{Lemma}
\begin{proof}
Assume that $\widetilde{S}\in C((T_\psi)^r)$. Since $T_\psi\in C((T_\psi)^r)$, we have
$
(T_\psi)^{-1}\circ\widetilde{S}\circ T_\psi\in C((T_\psi)^r)$.
Since $(T_\psi)^{-1}\circ\widetilde{S}\circ T_\psi\in C((T_\psi)^r)$ is a lift of $S$ and $(T_\psi)^r$ is ergodic, it follows by \cref{centrm} that
$$
(T_\psi)^{-1}\circ\widetilde{S}\circ T_\psi=\widetilde{S}\circ \tau_g\text{ for some }g\in G.
$$
Therefore
$$
(T_\psi)^{-2}\circ\widetilde{S}\circ (T_\psi)^2=(T_\psi)^{-1}\circ\widetilde{S}\circ \tau_g \circ T_\psi=(T_\psi)^{-1}\circ\widetilde{S}\circ T_\psi\circ\tau_g=\widetilde{S}\circ \tau_{g^2}
$$
and, in a similar way, $(T_\psi)^{-m}\circ\widetilde{S} \circ(T_\psi)^m=\widetilde{S}\circ \tau_{g^m}=\widetilde{S}$, i.e.\ $\widetilde{S}\in C((T_\psi)^m)$. Let $a,b\in\Z$ be such that $am+br=1$. We conclude that $\widetilde{S}$ commutes with $(T_\psi)^{am+br}=T_\psi$ which completes the proof.\end{proof}

\begin{Prop} \label{pcent}
Assume that $T_{\psi}$ has continuous spectrum on the orthocomplement of $L^2\xbm\ot \raz_G$. Suppose that $r\geq2$ is such that $T^r$ is ergodic and ${\rm gcd}(m,r)=1$. Then $C(T_\psi)=C((T_\psi)^r)$.\end{Prop}
\begin{proof}
Since, by assumptions, $(T_\psi)^r$ is ergodic, the assertion follows from \cref{centpot}.
\end{proof}

We can now formulate a general criterion
concerning the validity of Sarnak's conjecture for continuous finite group extensions.

\begin{Prop}\label{fgeSarnak} Let $\ov{T}$ be a uniquely ergodic homeomorphism which is a continuous extension of an odometer $T$, measure-theoretically isomorphic to $T$. Assume that its (rational discrete) spectrum is determined by finitely many prime numbers.  Assume that $\psi\colon X\to G$ is a cocycle with $G$ finite,~\eqref{uciagl} is satisfied, and $T_\psi$ has continuous spectrum in the orthocomplement of $L^2(X,\cb,m_X)\ot \raz_G$.
Assume moreover that the centralizers for all normal natural factors $T_{\psi H}$ of $T_\psi$ are $G/H$-trivial whenever $H\neq G$. Then, for each  $f\in C(X)$ and $j\in C(G)$ of zero mean,~\eqref{sarnakcon} is satisfied for $\ov{T}_{\ov{\psi}}$ and $(f\circ\pi)\ot j\in C(\ov{X}\times G)$ at each point.
\end{Prop}
\begin{proof} Fix $r,s$ two different prime numbers sufficiently large (so that $(T_\psi)^r$ and $(T_\psi)^s$ are ergodic). Notice that $T^r$ is then isomorphic to $T^s$.  Following \cref{mkm} (applied to $T^r$ isomorphic to $T^s$, both isomorphic to $T$) and \cref{zastmkm},
we first will prove  that if $H_1, H_2$ are proper normal subgroups of $G$ then $(T_{\psi H_1})^r$ is not isomorphic $(T_{\psi H_2})^s$.
For this aim, it is enough to notice is that $(T_{\psi H_2})^s$ cannot have an $r$-th root. Indeed, using the fact that the centralizer of $(T_{\psi H_2})^s$ is $G/H_2$-trivial and \cref{pcent}, if $(T_{\psi H_2})^s$ an $r$-th root then
$$
(T_{\psi H_2})^s=((T_{\psi H_2})^k\circ\tau_{g H_2})^r=(T_{\psi H_2})^{kr}\circ\tau_{g^r H_2}.
$$
It follows that $\tau_{g^r H_2}=(T_{\psi H_2})^{s-kr}$, which is an absurd as $s,r$ are prime ($s\neq kr$) and  $T_\psi$ is aperiodic.

Take any $(\ov{x},g)$. By the first part of the proof and \cref{narz}, we obtain
$$
\frac1N\sum_{n\leq N}\delta_{(\ov{T}_{\ov{\psi}})^{rn}\times (\ov{T}_{\ov{\psi}})^{sn}((\ov{x},g),(\ov{x},g))}\to \widetilde{\ov{\rho}}.
$$
Therefore
\begin{multline*}
\frac1N\sum_{n\leq N}((f\circ\pi)\ot j)(\ov{T}_{\ov{\psi}})^{rn}\times (\ov{T}_{\ov{\psi}})^{sn}((\ov{x},g),(\ov{x},g))\to\int ((f\circ\pi)\ot j)\cdot\ov{((f\circ\pi)\ot j)}\, d\widetilde{\ov{\rho}}\\
=\int (f\ot j)\cdot\ov{(f\ot j)}\,d \widetilde{\mu_R}=\int _X f\cdot \ov{f\circ R}\,dm_X\cdot \int_{G\times G}j\ot\ov{j}\,dm_G\ot m_G=0,
\end{multline*}
where the last equality follows by the assumption on $j$.
The result follows by \cref{kbsz}.
\end{proof}
\begin{Remark}\label{fgeSarnak1}
The assertion of \cref{fgeSarnak} remains true if in the orthocomplement of $L^2(X,\cb,m_X)\ot \raz_G$ there are finitely many rational eigenvalues (in the proof we need to exclude finitely many $r,s$).
\end{Remark}

\subsection{Special case: 2-point extensions of odometers}
We now consider the special case when $G=\Z/2\Z$. As an immediate consequence of \cref{mkm}, we obtain the following:
\begin{Cor}\label{lrs1} Let $T$ be an odometer and let $\phi,\psi\colon X\to\Z/2\Z$ be ergodic cocyles. Then,
either $T_\phi$ and $T_\psi$ are isomorphic or they are relatively disjoint over $T$, i.e.\ $J^e(T_\phi,T_\psi)=\{\widetilde{{(m_X)}_R} : R\in C(T)\}$.
\end{Cor}
\begin{Remark}
We give now a direct proof of \cref{lrs1}. Fix $\rho\in J^e(T_\phi,T_\psi)$. We have $(T_\phi\times T_\psi, \rho)\simeq (T_{\phi\times \psi\circ R},\kappa)$ where $\kappa$ projects on $m_X$ and $R\in C(T)$. If $\phi\times \psi\circ R$ is ergodic, it follows by \cref{rue} that $\kappa=m_X\otimes (m_G\otimes m_G)$, so $\rho=\widetilde{{(m_X)}_R}$. If $\phi\times \psi\circ R$ is not ergodic, then
$$
\phi-\psi\circ R=\xi-\xi\circ T \text{ for some measurable }\xi\colon X\to \Z/2\Z.
$$
It follows that $T_\phi$ and $T_\psi$ are isomorphic: $R_\xi \circ T_\phi = T_\psi \circ R_\xi$.
\end{Remark}

We also have the following (cf.\ \cref{pcent}).

\begin{Cor}\label{era3} Let $T$ be an odometer and let $\psi\colon X\to \Z/2\Z$ be ergodic. Assume that $T_\psi$ has continuous spectrum in the orthocomplement of $L^2(X,\mathcal{B},m_X)\ot\raz_{\Z/2\Z}$ and $C(T_\psi)$ is $\Z/2\Z$-trivial. Assume, moreover, that $r\neq s$ are prime numbers such that $T^r$ and $T^s$ are ergodic. Then
$(T_\psi)^r$  and $(T_\psi)^s$ are not isomorphic.
\end{Cor}

Now, using \cref{lrs1}, the corresponding part of \cref{fgeSarnak} takes the following form.

\begin{Cor}\label{lrs2}
Let $\ov{T}$ be a uniquely ergodic homeomorphism which is a continuous extension of an odometer $T$, measure-theoretically isomorphic to $T$. Assume that its (rational discrete) spectrum is generated by finitely many prime numbers.  Assume that $\psi\colon X\to \Z/2\Z$ is a cocycle,~\eqref{uciagl} is satisfied, and $T_\psi$ has continuous spectrum in the orthocomplement of $L^2(X,\cb,m_X)\ot \raz_{\Z/2\Z}$. Assume that for sufficiently large prime numbers $r\neq s$, the automorphisms $(T_\psi)^r$ and $(T_\psi)^s$ are not isomorphic.
Then for each $f\in C(X)$ and $\raz\neq j\in \widehat{\Z/2\Z}$,~\eqref{sarnakcon} is satisfied for $\ov{T}_{\ov{\psi}}$ and $(f\circ\pi)\ot j$ at each point.
\end{Cor}

\section{Applications}\label{se5}

\subsection{Bijective substitutions}
Let $\theta\colon A\to A^\la$ be a bijective substitution with the corresponding bijections $\sigma_i\in\mathscr{S}_r$.
Let $C(\theta)$ denote the centralizer of the set $\{\sigma_i : i=0,\ldots,\la-1\}$ in $\mathscr{S}_r$. Assume that $\eta\in C(\theta)$. Then $\eta$ induces a map $\widetilde{\eta}$ (both on finite blocks over $A$ and on $A^{\Z}$) given by
$$
\widetilde{\eta}(y)[n]:=\eta(y[n])\text{ for each }n\in\Z.
$$
We claim that $\widetilde{\eta}(X(\theta))=X(\theta)$. Indeed, since $\eta(\sigma_i(0))=\sigma_i(\eta(0))$, it follows that
$$
\widetilde{\eta}(\theta^n(0))=\theta^n(\eta(0))
$$
and we use the transitivity of the action of the group $G$ generated by $\sigma_0,\dots,\sigma_{\la-1}$ on $A$. Since $\widetilde{\eta}$ commutes with the shift, $\widetilde{\eta}\in C(S,X(\theta))$ (indeed, $(S,X(\theta))$ is uniquely ergodic, so $\widetilde{\eta}$ must preserve the unique measure). Now, the result from \cite{MR932072} shows that this is the only way to get non-trivial elements in the centralizer of the (measure-theoretic) dynamical system determined by a bijective substitution:

\begin{Th}[\cite{MR932072}] \label{leme}
$C(S,X(\theta))=\{S^i \circ \widetilde{\eta} : i\in \Z, \eta\in C(\theta)\}$.
\end{Th}

Suppose now that $\theta\colon G\to G^\lambda$ is a group substitution, i.e.\ $\theta(g)=\theta(e)\circ g$ for $g\in G$. Notice first that in each column~$j$ of the matrix for $\theta$, we have elements  $\theta(e)[j]\cdot g$. Therefore
$$
\mbox{$\sigma_j$ is the left translation on $G$ by $\theta(e)[j]$}
$$
and the group generated by $\theta(e)[j]$, $j=0,\ldots,\la-1$, is $G$. It follows that the group generated by $\sigma_0,\dots, \sigma_{\la-1}$ is the group of all left translations on $G$. Its centralizer $C(\theta)$ is equal to the group of all right translations. Thus, we obtain the following consequence of \cref{leme}:
\begin{Cor} \label{lemegr}
The centralizer of the group substitutions is $G$-trivial.
\end{Cor}

\begin{Remark}\label{nowauwaga}
The discrete part of the spectrum of the dynamical system $(S,X(\theta))$, where $\theta$ is a substitution, consists of the spectrum of the underlying odometer and a cyclic group determined by the \emph{height} $h$ of the substitution~\cite{MR0461470}. It follows that when the height is equal to 1, then the spectrum is continuous in the orthocomplement of the $L^2$-space of the underlying odometer. Otherwise, in this orthocomplement we have the cyclic group of eigenvalues generated by $e^{2\pi i /h}$.
\end{Remark}

We are now ready to show that Sarnak's conjecture holds for dynamical systems given by bijective substitutions.

\begin{Th}\label{main}
For each bijective substitution $\theta\colon A\to A^r$,
each function $F\in C(X(\theta))$, each bounded by~1, aperiodic multiplicative function $\lio\colon \N\to\C$ and each $y\in X(\theta)$, we have
\beq\label{tra}
\frac1N\sum_{n\leq N}F(S^ny)\lio(n)\to0\;\;\mbox{when}\;\;N\to\infty.\eeq
In particular, each topological dynamical system determined by a bijective substitution satisfies Sarnak's conjecture.
\end{Th}
\begin{proof}
It follows by \cref{lgc1} that it is enough to prove \eqref{tra} for the dynamical system $(S,X(\ov{\theta}))$ corresponding to the group cover substitution $\ov{\theta}$ of $\theta$. Moreover, in view of \cref{lmo2}, we can study instead its topologically isomorphic model $(S_{\va},X(\widehat{x})\times G)$.

Fix $f\ot j$ with $f\in C(X(\widehat{x}))$, where $j\in C(G)$, $\int j\,dm_G=0$. In view of \cref{kbsz}, \cref{lmo3a}, \cref{lemegr} and \cref{fgeSarnak}, \cref{fgeSarnak1} and \cref{nowauwaga}, for each $(y,g)\in X(\widehat{x})\times G$, we have
\begin{equation}\label{jeden}
\frac1N\sum_{n\leq N}(f\ot j)((S_{\va})^n(y,g))\lio(n)\to0
\end{equation}
for each multiplicative function $\lio$, $|\lio|\leq1$. If we now fix $\lio$,  then we have the relevant convergence (against this fixed $\lio$) for a linearly dense set of functions in $C(X(\widehat{x})\times G)$, hence for all functions in $C(X(\widehat{x})\times G)$ and the result follows.
\end{proof}

\begin{Remark}
Let  $j:=\raz_G$. Using \eqref{jeden}, for each $y\in X(\widehat{x})$, we have
$$
\frac1N\sum_{n\leq N}f(S^ny)\lio(n)\to0
$$
for each bounded by~1, aperiodic multiplicative function $\lio$. This can be also proved more directly. Notice that for each odometer $(T,X)$ we have~\eqref{sarnakcon} true with $\mob$ replaced by  $\lio$ since each finite system enjoys this property and $(T,X)$ is a topological inverse limit of such systems. If $(\ov{T},\ov{X})$ is a uniquely ergodic topological extension of $(T,X)$, measure-theoretically isomorphic to $(T,X,m_X)$, we can apply Lemma~7 and Proposition~3 in~\cite{Abdalaoui:2013rm} to lift the orthogonality condition~\eqref{tra} from the odometer to $(\ov{T},\ov{X})$.
\end{Remark}

\begin{Remark}
The proof of Sarnak's conjecture also gives the following: whenever $(S,X(\theta))$ is a subshift given by a bijective substitution, for each ergodic powers $S^r$ and $S^s$, each point $(y,z)\in X(\theta)\times X(\theta)$ is generic (for an ergodic measure).
\end{Remark}

\begin{Remark}\label{rem57}
In part 5 of \cite{MR3287361}, Drmota explains how the method developed by Mauduit and Rivat in the proof of Theorem~1 of \cite{MaRi2013} can be applied to any bijective substitution (Definition 4.1 in \cite{MR3287361} of ``invertible $\lambda$-automatic sequence'' corresponds to our bijective substitution of constant length $\lambda$). In particular, Theorem 5.5 of \cite{MR3287361} says that any bijective substitution $\theta$ on the alphabet $A$ satisfies a Prime Number Theorem, i.e.\ that for any $a\in A$, $\lim_{N\to\infty}\frac1{\pi(N)}|\{1\leq p<N: p\text{ is prime, }x_\theta[p]=a\}|$ exists (here $x_\theta$ stands for a fixed point given by $\theta$). In \cite{MaRi2013}, the proof of Theorem~2 can be deduced from the proof of Theorem~1 just by replacing the classical Vaughan identity by the similar result for the M\"obius function (see (13.39) and (13.40) in \cite{MR2061214}). It follows from this remark that the arguments given by Drmota in \cite{MR3287361} in order to generalize Theorem~1 from \cite{MaRi2013} to any bijective substitution is still valid to generalize Theorem~2 from \cite{MaRi2013} to any bijective substitution. This means that if $\theta$ is a bijective substitution and $J\colon A\to\C$ then
$$
\frac1N\sum_{n\leq N}J(x_\theta[n])\mob(n)\to 0.$$
\end{Remark}

\subsection{Regular Morse sequences and the Rudin-Shapiro case}\label{se:Ru-Sha}

In \cite{MR826355}, it has been proved that the centralizer of the dynamical systems given by so called \emph{regular Morse sequences} \cite{MR665892}  $x=b^0\times b^1\times\ldots$ ($b^t\in\{0,1\}^{\la_t}$, $t\geq0$) is $\Z/2\Z$-trivial.

\begin{Cor} If $x=b^0\times b^1\times\ldots$ is a regular Morse sequence for which
the set $\{p : p\text{ is prime and }p|\lambda_t \text{ for some }t\}$ is finite. Then~\eqref{tra} holds in the dynamical system given by $x$.
\end{Cor}
\begin{proof}
The result follows
from Corollaries~\ref{era3} and~\ref{lrs2} and the proof of \cref{main}.
\end{proof}

In \cite{MR937955}, the Rudin-Shapiro type sequences are
considered. These are {0-1}-sequences $x\in\{0,1\}^{\N}$ such that
$x[n]$ is equal to the mod~2  frequency of the block $1\ast\ldots\ast1$ (with fixed number of $\ast$) in the block given by the binary expansion on $n$.\footnote{We can also consider $1\ast\ldots\ast0$.} As shown in \cite{MR937955}, the corresponding subshift is given by a Toeplitz type $\Z/2\Z$-extension of the dyadic odometer, and the whole method applies.

\begin{Cor}\label{rusha}
If $x$ is a Rudin-Shapiro type sequence then~\eqref{tra} holds in
$(S,X(x))$. In particular, Sarnak's conjecture holds in the dynamical system given by the classical Rudin-Shapiro sequence.\end{Cor}
\begin{proof} Since $(S,X(x))$ has the
Lebesgue component of multiplicity  $2^k$ in the spectrum in the orthocomplement of the space generated by eigenfunctions~\cite{MR937955}, it follows that its $s$th and $r$th power also have Lebesgue components in the spectrum, of multiplicity $s2^k$ and $r2^k$, respectively. Thus, these powers cannot be isomorphic, unless $s=r$. The result follows from \cref{lrs2}.
\end{proof}

\section{Spectral approach and other methods}\label{se6}

Let $(S,X)$ with $X\subset A^\Z$ be a subshift over a finite alphabet $A$ with $|A|=r\geq 2$.

\subsection{First remarks}
\begin{Lemma}\label{lm:1A}
Suppose that~\eqref{sarnakcon} holds for arbitrary $x\in X$, for each function $f=\raz_B$, where $B\in A^k$ is a block of finite length ($k\geq 1$ is arbitrary) that appears on $X$. Then Sarnak's conjecture holds for $(S,X)$.
\end{Lemma}
\begin{proof}
It suffices to show~\eqref{sarnakcon} for a linearly dense family of functions in $C(X)$: e.g.\ functions which depend on a finite number of coordinates. The space of (continuous) functions depending on coordinates $[-k,k]$ in the full shift has dimension $r^{2k+1}$, which is at the same time the number of possible blocks of length $2k+1$. In a similar way, for a subshift, we just need to count the number of distinct $(2k+1)$-blocks appearing on $X$. Moreover, the family of their characteristic functions is linearly independent.
\end{proof}

\begin{Remark}\label{rss1}
There are other choices of finite families of functions than those in \cref{lm:1A} which also yield the validity of Sarnak's conjecture. For example, when $r=2$ we can use the so called Walsh basis: for each $K\geq1$, we consider
the characters of the group $\{0,1\}^{2K+1}$: $f_C(x)=(-1)^{\sum_{i\in C}x[i]}$ for $C\subset\{-K,\dots,K\}$ and $x\in X\subset A^\Z$.
\end{Remark}
\begin{Remark}\label{UWAGA:2}
In  \cite{MR3043150,MR2981162} the convergence in~\eqref{sarnakcon} is proved at any point for $f(y)=(-1)^{y[0]}$ ($f=f_{\{0\}}$ in the notation from \cref{rss1}) for Kakutani sequences.\footnote{The uniformity of estimates in these papers yields indeed~\eqref{sarnakcon} at any point $y\in X(x))$.}  A natural question arises whether this is sufficient to obtain Sarnak's conjecture for the corresponding dynamical system. In general, it does not seem to be automatic that~\eqref{sarnakcon} for functions depending on one coordinate implies~\eqref{sarnakcon} for functions depending on more coordinates. E.g., in~\cite{Abdalaoui:2013rm}, where Sarnak's conjecture is proved for the $0$-$1$-subshift generated by the Thue-Morse sequence, \eqref{sarnakcon} for $f(y)=(-1)^{y[0]}$ is proved by completely different methods than for continuous functions invariant under the map $y\mapsto \widetilde{y}$, where $\widetilde{y}[n]=1-y[n]$.\footnote{E.g.\ $g(y)=(-1)^{y[0]+y[1]}$ is invariant under this map; notice that $g=f_{\{0,1\}}$ from \cref{rss1}.} We note that the method from \cref{lrs1} does not apply to Kakutani systems since their centralizer can be uncountable: there are Kakutani sequences for which the corresponding dynamical systems are rigid~\cite{MR795787}. However, in \cref{spektralna} we provide an argument which in \cref{aplikacje} will be used to show that~\cite{MR3043150,MR2981162} yield Sarnak's conjecture for the dynamical systems given by Kakutani sequences.
\end{Remark}

\begin{Remark}
We have already shown that Sarnak's conjecture holds for the dynamical system given by the Rudin-Shapiro sequence, see~\cref{rusha}. Recall also that in this case~\eqref{sarnakcon} was shown earlier in~\cite{MaRi2013} for $f(y)=(-1)^{y[0]}$ (at any point). Here the situation is more delicate if we want to apply the method from \cref{spektralna}: we need more functions, see \cref{aplikacje} for more details.
\end{Remark}

\subsection{Spectral approach}\label{spektralna}
\begin{Lemma}\label{lmrs17} Assume that $T$ is a uniquely ergodic homeomorphism of a compact metric space $X$. Denote the unique $T$-invariant measure by $\mu$. Assume that the unitary operator $U_T\colon L^2\xbm\to L^2\xbm$, $U_Tg:=g\circ T$, has simple spectrum. Assume that the  maximal spectral type of $U_T$ is realized by $F\in C(X)$.  If  $F$ satisfies~\eqref{sarnakcon} at each
point $x\in X$ then $(T,X)$ satisfies Sarnak's conjecture.
\end{Lemma}
\begin{proof}
Observe first that if $F$ satisfies~\eqref{sarnakcon} at each point then the same is true for each function $p(U_T)F$  of $F$ (where $p(z)=\sum_{\ell=-K}^Ka_\ell z^\ell$ is a trigonometric polynomial). By the simplicity of the spectrum of $U_T$, the set of functions of the form $p(U_T)F$ is dense in $L^2\xbm$. We now repeat the argument from Lemma~7 in \cite{Abdalaoui:2013rm}. Fix $G\in C(X)$, $x\in X$ and $\vep>0$. Find a trigonometric polynomial $p$ so that $\|p(U_T)F-G\|_2<\vep$. Let $N_0$  be such that for $N\geq N_0$, $|\frac1N\sum_{n\leq N}p(U_T)F(T^nx)\mob(n)|<\vep$ for all $N>N_0$. Then (since $T$ is uniquely ergodic and $|\mob|\leq 1$)
\begin{align*}
&\left|\frac1N\sum_{n\leq N}G(T^nx)\mob(n)\right|\\
&\leq \frac1N\sum_{n\leq N}|(G-p(U_T)F)(T^nx)||\mob(n)|
+ \left|\frac1N\sum_{n\leq N}(p(U_T)F)(T^nx)\mob(n)\right|\\
&\leq\frac1N\sum_{n\leq N}|(G-p(U_T)F)(T^nx)|+\vep\to \|G-p(U_T)F\|_1+\vep
\end{align*}
when $N\to\infty$. Since $\|G-p(U_T)F\|_1\leq\|G-p(U_T)F\|_2<\vep$, the result follows.
\end{proof}

\begin{Remark} \label{uwfunkcja}\
\begin{enumerate}[(A)]
\item
The assertion of \cref{lmrs17} remains true if we take any bounded arithmetic function $\lio\colon\N\to\C$ instead of $\mob$ (both in~\eqref{sarnakcon} and in Sarnak's conjecture). The proof is the same.
\item
Fr\k{a}czek~\cite{MR1444806} showed that for each automorphism $T$ on $\xbm$, where $X$ is a compact metric space, the maximal spectral type of  $U_T$ is always realized by a continuous function. However, in order to prove Sarnak's conjecture using \cref{lmrs17}, we look for \emph{natural} continuous functions realizing the maximal spectral type for which we can show that~\eqref{sarnakcon} holds.

We would like to mention also an open problem raised by Thouvenot in the 1980th whether each ergodic dynamical system has $L^1$-simple spectrum, i.e., for some $f\in L^1\xbm$, we have ${\rm span}\{f\circ T^k:\:k\in\Z\}$ dense in $L^1\xbm$. If the answer to Thouvenot's problem is positive and the $L^1$-simplicity can be realized by a continuous function $f$ in each uniquely ergodic system, then (cf.\ the proof of \cref{lmrs17})  to prove Sarnak's conjecture we need to check~\eqref{sarnakcon} for $f$ (at each point).
\item
Suppose that the continuous and discrete part of the maximal spectral type of $U_T$ are realized by $f\in C(X)$ and $g\in C(X)$, respectively. Then, by elementary spectral theory, $F=f+g\in C(X)$ realizes the maximal spectral type of $U_T$ and, clearly, it suffices to check that \eqref{sarnakcon} holds both for $f$ and $g$ (at each point) to see that it holds for $F$ (at each point).
\item
\cref{lmrs17} has a natural extension to uniquely ergodic homeomorphisms $T$ such that $U_T$ has non-trivial multiplicity. All we need to know is that $L^2\xbm$ has a decomposition into cyclic spaces: $L^2\xbm=\bigoplus_{k\geq1}\Z(f_k)$ with $f_k\in C(X)$ and check~\eqref{sarnakcon} for these generators.\footnote{Recall however that it is open whether for an arbitrary automorphism $T$ on $\xbm$, where $X$ is a compact metric space there are continuous functions $f_k$, $k\geq 1$ such that $L^2\xbm=\bigoplus_{k\geq1}\Z(f_k)$ and $\sigma_{f_{k+1}}\ll\sigma_{f_k}$, $k\geq1$, see e.g.\ \cite{MR2576267} for more details.} We will find such functions in the next section in case of the dynamical systems given by the Rudin-Shapiro type sequences.\footnote{Recall that in the general case of dynamical systems given by the Rudin-Shapiro type sequences, this multiplicity is of the form $2^k$, $k\geq 1$~\cite{MR937955}.}
\end{enumerate}
\end{Remark}

\subsection{Applications}\label{aplikacje}
\paragraph{Generalized Morse sequences over $A=\{0,1\}$}

\begin{Prop}
Let $x$ be a generalized Morse sequence over $A=\{0,1\}$. Then Sarnak's conjecture holds for $(S,X(x))$ if and only if~\eqref{sarnakcon} holds (at each point) for $f(y)=(-1)^{y[0]}$.
\end{Prop}
\begin{proof}
Recall that $(S,X(x))$ is uniquely ergodic~\cite{MR0239047} (with the unique invariant measure $\mu_x$) and has simple spectrum~\cite{MR638749}. As proved in \cite{MR638749} (see also \cite{MR1685402}) $f(y)=(-1)^{y[0]}$ realizes the continuous part of the maximal spectral type of $U_S$ (on $L^2(X(x),\mu_x)$). Moreover, the discrete part is given by the equicontinuous factor of $(S,X(x))$, which is the odometer determined by $\la_t$, $t\geq 0$. It follows that the eigenfunctions $g_i$, $i\geq1$, are continuous.  If $g=\sum_{i\geq 1}a_ig_i$, each $a_i\neq 0$, $\sum_{i\geq1}|a_i|<+\infty$, then $g$ is a continuous function realizing the discrete part of the maximal spectral type of $U_S$. Since each odometer is a topological inverse limit of systems defined on finitely many points, and for finite systems Sarnak's conjecture holds because of the PNT in arithmetic progressions, therefore $g$ satisifes~\eqref{sarnakcon}. Thus, in view of \cref{lmrs17} and \cref{uwfunkcja} (C), it suffices to prove~\eqref{sarnakcon} for $f$ to obtain the validity of Sarnak's conjecture.
\end{proof}
\begin{Cor}\label{ten:winosek}
Sarnak's conjecture holds for the dynamical systems given by Kakutani sequences.
\end{Cor}
\begin{proof}
In view of the above proposition, it suffices to prove~\eqref{sarnakcon} for $f$ which was done in~\cite{MR3043150,MR2981162} (cf.\ \cref{UWAGA:2}).
\end{proof}

\paragraph{Rudin-Shapiro type sequences}

Recall that the classical Rudin-Sharpiro sequence $x\in\{0,1\}^{\N}$ is defined in the following way:
\begin{itemize}
\item
take the fixed point of the substitution $a\mapsto ab$, $b\mapsto ac$, $c\mapsto db$ and $d\mapsto dc$,
\item
use the code $a,b\mapsto 0$ and $c,d\mapsto 1$ to pass to the space of {0-1}-sequences (the map arising from this code yields a topological isomorphism of the relevant subshifts on four and two letters).
\end{itemize}
The multiplicity of the corresponding dynamical system on the continuous part of the spectrum is equal to~2~\cite{MR937955,MR2590264}. It follows from \cite{MR937955} and \cref{uwfunkcja}~(D) that, in order to obtain Sarnak's conjecture for the corresponding subshift, we need to check~\eqref{sarnakcon} for two continuous functions
 $f\cdot \raz_{D^2_0}$ and $f\cdot \raz_{D^2_1}$ (cf.\ \cref{f13}). It follows immediately from the definition of $D^2_0$ and the recognizability of substitutions \cite{Mosse:1992}
that $\raz_{D^2_0}$ is a continuous function depending on a finite number of coordinates. Therefore, to obtain Sarnak's conjecture for the subshift given by the Rudin-Shapiro sequence, we would have to check~\eqref{sarnakcon} for the elements of the Walsh basis of order $4M$, where $M$ is the constant of recognizability, see~\cref{f13a}. Recall that~\eqref{sarnakcon} was already shown for $f$ in \cite{MaRi2013,Ta1}. Notice that this approach to prove Sarnak's conjecture is completely different from the one presented in the preceding sections (cf.\ \cref{rusha}). The above applies to all Rudin-Shapiro type sequences.

\subsection{Comparison with results of Veech \cite{Ve}}\label{ostatnia1}
In the recent preprint \cite{Ve}, Veech considers a class of systems for which Sarnak's conjecture holds. We will now briefly present his work and then compare it with our results. Assume that $\la_n\geq2$ for $n\geq0$, then set $n_0:=1$ and $n_t:=\prod_{k=0}^{t-1}\lambda_k$, $t\geq 1$ and define
$$
X:={\rm lim inv}_{t\to\infty} \Z/n_t\Z=\{x=(x_t) : 0\leq x_t<n_t, x_{t+1}=x_t\bmod n_t, t\geq1\}.$$
This is a compact, Abelian, monothetic group on which we consider $Tx=x+\theta$ with $\theta=(1,1,\ldots)$.
It is not hard to see that the systems obtained this way are naturally isomorphic to the odometers considered in \cref{se:morsecoc}.
The sequence of towers $\mathcal{D}^t$, $t\geq1$, in the new coordinates is determined by
$$
D^t_0:=\{x\in X :  x_t=0\},$$
and we obtain  pairwise disjoint sets $D^t_0,TD^t_0,\ldots T^{n_t-1}D_0^t$ with $\bigcup_{j=0}^{n_t-1}T^jD^t_0=X$.
Then define
$$
\tau(x):=\min\{t\geq 1: x_{t}\neq n_{t}-1\}.
$$
We have $\lim_{x\to -\theta}\tau(x)=\infty$ and $\tau$ is continuous on $X\setminus\{-\theta\}$.

Let $K$ be a compact group and take $(\Psi(t))_{t\geq1}\subset K$. Set
$$
f(x):=\Psi(\tau(x)).$$
Then $\Psi$ is locally constant on $X\setminus\{-\theta\}$ and $f\in C(X\setminus\{-\theta\},K)$. There are some assumptions on the sequence $\Psi$ made in \cite{Ve}:
\begin{enumerate}[(i)]
\item
\text{$\lim_{t\to\infty} \Psi(t)$ does not exist},\label{assu-ve}
\item\label{c2}
$\{\Psi(t) : t\geq 1\}$ generates a dense subgroup of $K$, and so does the set
$\{\Psi(t)\Psi(u)^{-1} : t,u\geq 1\}$,
\item\label{c3}
$(\Psi(t))_{t\geq 1}$ is recurrent (that is, every initial block of $\Psi$ repeats infinitely often).
\end{enumerate}
\begin{Remark}
If $K=\Z/2\Z$, the conditions \eqref{c2} and \eqref{c3} are not necessary.
\end{Remark}

Let $M_\Psi\subset K^\Z$ be the closure of all sequences $(f(x+n\theta)_{n\in\Z}$ for $x\in X\setminus\Z\theta$. On $M_\Psi$, we consider the usual shift $S$. Let
$$
m\colon M_\Psi\to K\text{ be given by } m(y)=y[0].
$$
Finally, let $S_m\colon M_\Psi\times K\to M_\Psi\times K$ be the skew product defined as $S_m(y,k)=(Sy,m(y)k)$. Then $S_m$ is a homeomorphism of $M_\Psi\times K$.
\begin{Th}[\cite{Ve}]\label{twvee}
Suppose additionally that the set $\{\lambda_t : t\geq 0\}$ is finite. Then, under the above assumptions, $(S_m,M_\Psi\times K)$ satisfies Sarnak's conjecture.
\end{Th}

\begin{Remark}
It is not hard to see that in the language of \cite{MR2180227}, the function $f$ is a semicocycle over an odometer (that is, a function continuous on a residual subset of an odometer). It follows from \cite{MR2180227} that the dynamical system given by $(S,M_\Psi)$ is a Toeplitz dynamical system (cf.\ \cref{se:morsecoc}). The system is regular \cite{MR2180227}, hence uniquely ergodic and measure-theoretically isomorphic to the odometer $(T,X)$.

Notice that if $x_t=n_t-1$ then also $x_{k}=n_{k}-1$ for $1\leq k\leq t-1$. It follows immediately that $\tau$ is constant on each $D_i^t$, $0\leq i\leq n_t-2$. Therefore, the cocycle $f\colon X\to K$ defined above is a Morse cocycle (cf.\ \cref{se:morsecoc}). It has the following additional property:
\begin{equation}\label{condj}
f \text{ is constant on }\bigcup_{1\leq j\leq \lambda_{t+1}-1}D_{jn_t-1}^{t+1}.
\end{equation}
Condition~\eqref{condj} yields the class of Morse sequences $x=b^0\times b^1 \times \dots$, where each block $b_t$, $t\geq 0$, is of the form $b^t=ek\ldots k^{|b^t|-1}$ (cf.\ \cref{uw16}). In particular, if $K=\Z/2\Z$, we have $x=b^0\times b^1 \times \dots$, where $b^t=0\ldots0$, $b^t=01\ldots 01$ or $b^t=010\ldots 10$, $t\geq 0$. Notice that Kakutani sequences are of this form.

Notice also that $(S_m,M_\Psi\times K)$ corresponds to $(S_\varphi,X(\widehat{x})\times G)$ defined in \cref{se:morsecoc}, which is, in turn, topologically isomorphic to $(S,X(x))$.

It follows that \cref{twvee} is a significant extension of the validity of Sarnak's conjecture  for Kakutani systems (cf.\ \cref{ten:winosek}) to general compact groups.
\end{Remark}

\small
\bibliography{cala.bib}

\bigskip
\footnotesize

\noindent
S\'ebastien Ferenczi\\
\textsc{Institut de Math\'ematiques de Marseille, CNRS - UMR 7373}\\
\textsc{Case 907 - 163, av.\ de Luminy, F13288 Marseille Cedex 9, France}\\
\noindent
\textit{E-mail address:} \texttt{ssferenczi@gmail.com}

\medskip

\noindent
Joanna Ku\l aga-Przymus\\
\textsc{Institute of Mathematics, Polish Acadamy of Sciences, \'{S}niadeckich 8, 00-956 Warszawa, Poland}\\
\textsc{Faculty of Mathematics and Computer Science, Nicolaus Copernicus University, Chopina 12/18, 87-100 Toru\'{n}, Poland}\par\nopagebreak
\noindent
\textit{E-mail address:} \texttt{joanna.kulaga@gmail.com}

\medskip

\noindent
Mariusz Lema\'{n}czyk\\
\textsc{Faculty of Mathematics and Computer Science, Nicolaus Copernicus University, Chopina 12/18, 87-100 Toru\'{n}, Poland}\par\nopagebreak
\noindent
\textit{E-mail address:} \texttt{mlem@mat.umk.pl}

\medskip

\noindent
Christian Mauduit\\
\textsc{Institut de Math\'ematiques de Marseille, CNRS - UMR 7373}\\
\textsc{Case 907 - 163, av.\ de Luminy, F13288 Marseille Cedex 9, France}\\
\noindent
\textit{E-mail address:} \texttt{mauduit@iml.univ-mrs.fr}

\end{document}